\documentclass[a4paper,12pt]{article}

\usepackage{amsmath}
\usepackage{amsthm}
\usepackage{amssymb}
\usepackage{cancel}

\usepackage{epsfig}
\usepackage{psfrag}
\usepackage{subfigure}
\usepackage{graphicx}

\usepackage[mathscr,mathcal]{eucal}
\usepackage{enumerate}

\usepackage{t1enc}
\usepackage[latin2]{inputenc}

\def \C {\mathbb C}

\def \R {\mathbb R}

\def \Z {\mathbb Z}

\def\cB{\mathcal{B}}
\def\cC{\mathcal{C}}
\def\cD{\mathcal{D}}

\def\cF{\mathcal{F}}

\def\cH{\mathcal{H}}

\def\cK{\mathcal{K}}
\def\cL{\mathcal{L}}

\def\cS{\mathcal{S}}

\def\cU{\mathcal{U}}

\def\vareps{\varepsilon}

\def\vp{\mathbf{p}}

\def\vx{\mathbf{x}}
\def\vy{\mathbf{y}}

\newcommand{\expect}[1]{\ensuremath{\mathbf{E}\left(#1\right)}}

\newcommand{\cov}[2]{\ensuremath{\mathbf{Cov}\big(#1,#2\big)}}
\newcommand{\condprob}[2]{\ensuremath{\mathbf{P}\big(#1\bigm|#2\big)}}
\newcommand{\condexpect}[2]{\ensuremath{\mathbf{E}\big(#1\bigm|#2\big)}}

\newcommand{\ind}[1]{\ensuremath{{1\!\!1}_{\{#1\}}}}

\def\ordo{o}

\def\one{1\!\!1}

\def\grad{\mathrm{grad}\,}

\DeclareMathOperator*{\Perm}{Perm}

\renewcommand{\d}{\mathrm d}

\newcommand{\abs}[1]{\left|\,{#1}\,\right|}
\newcommand{\norm}[1]{\left\|\,{#1}\,\right\|}

\def \wt {\widetilde}
\def\ol{\overline}

\def\wh{\widehat}

\newtheorem {theorem}{Theorem}

\newtheorem {lemma}{Lemma}

\newtheorem {corollary}{Corollary}

\newtheorem {proposition}{Proposition}

\newtheorem* {theorem*}{Theorem}
\newtheorem* {thm*}{Theorem}
\newtheorem* {lemma*}{Lemma}
\newtheorem* {lem*}{Lemma}
\newtheorem* {corollary*}{Corollary}
\newtheorem* {cor*}{Corollary}
\newtheorem* {proposition*}{Proposition}
\newtheorem* {prop*}{Proposition}
\newtheorem* {definition*}{Definition}
\newtheorem* {def*}{Definition}
\newtheorem* {conjecture*}{Conjecture}
\newtheorem* {remark*}{Remark}
\newtheorem* {rem*}{Remark}

\theoremstyle{definition}
\newtheorem*{remarks}{Remarks}
\newtheorem*{remark}{Remark}
\newtheorem*{ack}{Acknowledgement}

\def\be{\begin{equation}}
\def\ee{\end{equation}}

\def\bea{\begin{eqnarray}}
\def\eea{\end{eqnarray}}

\newcommand{\wick}[1]{\ensuremath{:\!\! #1 \!\!:\,}}

\title{Diffusive limits for ``true'' (or myopic) self-avoiding random walks and self-repellent Brownian polymers in $d\ge3$}

\author{
{\sc Ill\'es Horv\'ath} \qquad {\sc B\'alint T\'oth} \qquad {\sc B\'alint Vet\H o}
\\
Institute of Mathematics, Budapest University of Technology
\\
Egry J\'ozsef u.\ 1, Budapest, H-1111, Hungary
\\
email: {\tt \{pollux,balint,vetob\}@math.bme.hu}
}

\begin{document}

\maketitle

\begin{abstract}
The problems considered in the present paper have their roots in two different
cultures. The \emph{`true' (or myopic) self-avoiding walk} model (TSAW) was
introduced in the physics literature by Amit, Parisi and Peliti in
\cite{amit_parisi_peliti_83}. This is a nearest neighbor non-Markovian random
walk in $\Z^d$ which prefers to jump to those neighbors which were less visited
in the past. The \emph{self-repelling Brownian polymer} model (SRBP), initiated
in the probabilistic literature by Durrett and Rogers in
\cite{durrett_rogers_92} (independently of the physics community), is the
continuous space-time counterpart: a diffusion in $\R^d$ pushed by the negative
gradient of the (mollified) occupation time measure of the process. In both
cases, similar long memory effects are caused by a path-wise self-repellency of
the trajectories due to a push by the negative gradient of (softened) local
time.

We investigate the asymptotic behaviour of TSAW and SRBP in the non-recurrent dimensions. First, we identify a natural stationary (in time) and ergodic distribution of the environment (the local time profile) as seen from the
moving particle. The main results are diffusive limits. In the case of TSAW, for a wide class of self-interaction functions, we establish diffusive lower and upper bounds for the displacement and for a particular, more  restricted
class of interactions, we prove full CLT for the finite dimensional distributions of the displacement. In the case of SRBP, we prove full CLT without restrictions on the interaction functions.

These results settle part of the conjectures, based on non-rigorous renormalization group arguments (equally 'valid' for the TSAW and SRBP cases), in \cite{amit_parisi_peliti_83}. The proof of the CLT follows the non-reversible version of Kipnis\,--\,Varadhan theory. On the way to the proof,
we slightly weaken the so-called \emph{graded sector condition}.


\end{abstract}

\section{Introduction and background}
\label{s:intro}

We consider two phenomenologically very similar, but, on a technical level, not
quite identical models of self-repelling random motion: the ``true'' (or
myopic) self-avoiding random walk (TSAW) on the integer lattice $\Z^d$ and the
self-repelling Brownian polymer (SRBP) in $\R^d$. Due to this, most sections
and subsections of this paper have two sub-units: one devoted to the TSAW and
another one to SRBP. The SRBP is technically similar to a particular choice of
the self-interaction function in TSAW, which we call the Gaussian case. We try
to avoid overlaps and repetitions of technical arguments. Hence, the proofs are
complete and self-contained for the TSAW model, but not for the SRBP. In this
second case, we mostly just sketch the proofs. A fully self-contained version
is available online in \cite{horvath_toth_veto_09}.

Due to this parallelism the notation of the paper has some peculiarity: we use
the same notation for analogous objects in the two cases. E.g.\ $X(t)$ will
denote the displacement of both random motions, $G$ will denote the
infinitesimal generator of the underlying main Markov processes in both cases,
etc. But we think that there is no danger of confusion, since the two models
are treated in separate subunits and it is always clear from the context which
model is under investigation. Further on, some particular symbols are used
``abusively'': $\nabla$ and $\Delta$ denote lattice gradient and Laplacian, but
also gradients and Laplacians acting on the Hilbert spaces $\cL^2(\Omega, \pi)$
where $(\Omega,\pi)$ is an appropriate probability space on which translations
act ergodically. The meaning of these symbols will be always clear from the
context. The main structure of the paper is as follows: In Section 1, (the rest
of the present section) we define the models, sketch the main problems and
their history. In Section 2, we make the necessary formal preparations and
formulate the main results of the paper. In Section 3, the appropriate
functional analytic formalism is prepared: the suitable $\cL^2$ Hilbert spaces
and the relevant linear operators are introduced and partly analyzed. In
Section 4, we give a survey of the non-reversible Kipnis\,--\,Varadhan theory
of central limit behaviour of additive functionals of ergodic Markov processes,
including a slight improvement in the so-called graded sector condition.
Section 5 contains the so-called $H_{-1}$-bounds (that is: diffusive bounds)
for the compensators appearing in the decomposition of displacement of the
random motions investigated. Finally, in Section 6, the graded sector condition
is verified.

\subsection{TSAW}
\label{ss:background_TSAW}

Let $w:\R\to (0,\infty)$ be a fixed smooth ``rate function'' for which
\begin{align}
\label{ellipticity}
\inf_{u\in\R}w(u):=\gamma>0,
\end{align}
and denote by $s$ and $r$ its even, respectively, odd part:
\begin{align}
\label{evenodd}
s(u):=\frac{w(u)+w(-u)}2 -\gamma,
\qquad
r(u):=\frac{w(u)-w(-u)}2.
\end{align}
Beside \eqref{ellipticity}, we make the following assumptions: there exist constants $c>0$, $\vareps>0$ and $C<\infty$ such that
\begin{gather}
\label{convexity} \inf_{u\in\R}r^{\prime}(u)> c,
\\
\label{s_small} s(u)< C\exp\{(c-\vareps) u^2/2\},
\end{gather}
and, finally, we make the technical assumption that $r(\cdot)$ is an entire function which satisfies:
\begin{align}
\label{r_entire}
\sum_{n=0}^\infty \left(\frac2{c}\right)^{n/2} \abs{r^{(n)}(0)}<\infty.
\end{align} Condition \eqref{ellipticity} is \emph{ellipticity} which ensures that the jump rates of the random walk considered are \emph{minorated} by an ordinary simple symmetric walk. Condition \eqref{convexity} ensures sufficient self-repellence of the trajectories and sufficient log-convexity of the stationary measure identified later. Conditions \eqref{s_small} and \eqref{r_entire} are of technical nature and their role will be clarified later.

Let $t\mapsto X(t)\in\Z^d$ be a continuous time nearest neighbor jump process on the integer lattice $\Z^d$ whose law is given as follows:
\begin{align}
\label{law}
\condprob{X(t+\d t)=y} {\cF_t,X(t)=x}= \ind{\abs{x-y}=1}
w(\ell(t,x)-\ell(t,y))\,\d t +\ordo(\d t)
\end{align}
where
\begin{equation}
\ell(t,z) := \ell(0,z) + \abs{ \{0\le s\le t: X(s)=z\} }
\qquad
z\in\Z^d
\end{equation}
is the occupation time measure of the walk $X(t)$ with some initial values $\ell(0,z)\in\R$, $z\in\Z^d$.  This is a continuous time version of the \emph{`true' self-avoiding random walk} defined in \cite{amit_parisi_peliti_83}.

Non-rigorous (but nevertheless convincing) scaling and renormalization group arguments suggest the following dimension-dependent asymptotic scaling behaviour (see e.g.\ \cite{amit_parisi_peliti_83}, \cite{obukhov_peliti_83},
\cite{peliti_pietronero_87}):

\begin{enumerate}[--]

\item
In $d=1$: $X(t)\sim t^{2/3}$ with intricate, non-Gausssian scaling limit.

\item
In $d=2$: $X(t)\sim t^{1/2}(\log t)^{\zeta}$ and Gaussian (that is Wiener)
scaling limit expected. (We note that actually there is some controversy in the
physics literature about the value of the exponent $\zeta$ in the logarithmic
correction.)

\item
In $d\ge3$: $X(t)\sim t^{1/2}$ with Gaussian (i.e.\ Wiener) scaling limit
expected.
\end{enumerate}

In $d=1$, for some particular cases of the model (discrete time TSAW with edge, rather than site repulsion and continuous time TSAW with site repulsion, as defined above), the limit theorem for $t^{-2/3}X(t)$ was established in \cite{toth_95}, respectively, \cite{toth_veto_09} with the truly intricate
limiting distribution identified. The scaling limit of the \emph{process} $t\mapsto N^{-2/3}X(Nt)$ was constructed and analyzed in \cite{toth_werner_98}.

In $d=2$, very little is proved rigorously. For the isotropic model exposed
above, we expect the value $\zeta=1/4$ in the logarithmic correction. For a
modified, anisotropic  version of the model where self-repulsion acts only in
one spatial (say, the horizontal) direction, the exponent $\zeta=1/3$ is
expected and the lower bound $\varliminf_{t\to\infty} t^{-1}(\log t)^{-1/2}
\expect{X(t)^2}>0$ is actually proved, cf.\ \cite{valko_09}.

In the present paper, we address the $d\ge3$ case.

\subsection{SRBP} \label{ss:background_SRBP}

Let $V:\R^d\to\R_+$ be an approximate identity, that is a smooth ($C^{\infty}$), spherically symmetric function with sufficiently fast decay at infinity (exponential decay certainly suffices), and
\begin{equation}
\label{FisgradV}
F:\R^d\to\R^d,
\qquad
F(x):=-\grad \, V(x).
\end{equation}
For reasons which will become clear later, we also impose the condition of \emph{positive definiteness} of $V$:
\begin{equation}
\label{Vposdef}
\wh V(p):=(2\pi)^{-d/2}\int_{\R^d} e^{i p\cdot x}V(x) \,\d x
\ge0.
\end{equation}
A particular choice could be $V(x):=\exp\{-\abs{x}^2\}$.

Let $t\mapsto B(t)\in\R^d$ be standard $d$ dimensional Brownian motion and
define the stochastic process $t\mapsto X(t)\in\R^d$ as the solution of the SDE
\begin{equation}
\label{Brpoly}
X(t)=  B(t)+\int_0^t\int_0^s F(X(s)-X(u))\,\d u\,\d s,
\end{equation}
or
\begin{equation}
\label{Brpolydiff} \d X(t)= \d B(t)+ \left(\int_0^t F(X(t)-X(u))\,\d u\right)\d
t.
\end{equation}

\begin{remark}
Other types of self-interaction functions $F$ in \eqref{Brpoly} and
\eqref{Brpolydiff} may give rise to various different asymptotics. For the very
few rigorous results (mostly in 1d), see \cite{norris_rogers_williams_87},
\cite{durrett_rogers_92}, \cite{cranston_lejan_95},
\cite{cranston_mountford_96} and in particular \cite{mountford_tarres_08} which
also contains a survey of the earlier results. Recent 1d results appear in
\cite{tarres_toth_valko_09}.
\end{remark}

Now, introducing the occupation time measure
\begin{equation}
\ell(t, A):=\ell(0,A) + \abs{\{0<s\le t: X(s)\in A\}}
\end{equation}
where $A\subset \R^d$ is any measurable domain, and $\ell(0,A)$ is some signed initialization, we can rewrite the SDE \eqref{Brpolydiff} as follows:
\begin{equation}
\label{Brpolydiff2} \d X(t)= \d B(t) - \grad \big(V*\ell
(t,\cdot)\big)(X(t))\,\d t
\end{equation}
where $*$ stands for convolution in $\R^d$. We assume that $\ell(0,A)$ is a signed Borel measure on $\R^d$ with slow increase: for any $\varepsilon>0$
\begin{equation}
\label{slowincreaseinitially}
\lim_{N\to\infty} N^{-(d+\vareps)}\abs{\ell}(0,[-N,N]^d)=0.
\end{equation}
The form \eqref{Brpolydiff2} of the driving mechanism, compared with
\eqref{law}, shows explicitly the phenomenological similarity of the two
models. The non-rigorous scaling and renormalization group arguments mentioned
earlier in the context of TSAW apply equally well to the SRBP, and thus, the
same dimension dependent behaviour is expected to hold. Beside the earlier
results for TSAW listed in the previous subsection, we should mention that
recently, a robust proof was given for the super-diffusive behaviour of the 1d
models: in \cite{tarres_toth_valko_09}, inter alia, it is proved that, for the
$1d$ SRBP models, $\varliminf_{t\to\infty}t^{-5/4}\expect{X(t)^2}>0$ and
$\varlimsup_{t\to\infty}t^{-3/2}\expect{X(t)^2}<\infty$. These are robust
super-diffusive  bounds (not depending on microscopic details), but still far
from the expected $t^{2/3}$ scaling.

\section{Formal setup and results} \label{s:setup}

For both problems, we consider the $d\ge3$ cases. First, we identify a rather
natural stationary and ergodic (in time) distribution of the environment
(essentially: the local time profile) as seen from the position of the moving
point. In this particular stationary and ergodic regime, we prove diffusive
(that is of order $t$) bounds on the variance of $X(t)$ and  \emph{diffusive
limit} (that is non-degenerate CLT with normal scaling) for the displacement.

\subsection{TSAW} \label{ss:setup_results_TSAW}

It is natural to consider the local time profile as seen from the position of the random walker
\begin{align}
\label{etadef}
\eta(t) = \big( \eta(t,x) \big)_{x\in\Z^d}
\qquad
\eta(t,x):=\ell(t,X(t)+x).
\end{align}
It is obvious that $t\mapsto\eta(t)$ is a c.a.d.l.a.g.\ Markov process on the
state space
\begin{align}
\label{OmegaTSAW} \Omega:=\{\omega=\big(\omega(x)\big)_{x\in\Z^d}\,:\,
\omega(x)\in\R,\,\,\, (\forall \vareps>0) \lim_{\abs{x}\to\infty}
\abs{x}^{-\vareps} \abs{\omega(x)}=0\}.
\end{align}
Note that we allow initial values $\ell(0,x)\in\R$ for the occupation time measure and thus $\ell(t,x)$ need not be non-negative. The group of spatial shifts
\begin{align}
\label{shifts} \tau_z:\Omega\to\Omega, \qquad \tau_z\omega(x):=\omega(z+x),
\qquad z\in\Z^d
\end{align}
acts naturally on $\Omega$.

Let
\begin{align}
\cU:=\{e\in\Z^d: \abs{e}=1\}.
\end{align}
Throughout the paper, we will denote by $e$ the $2d$ unit vectors from $\cU$
and by $e_l$, $l=1,\dots,d$, the unit vectors pointing in the positive
coordinate directions.

The infinitesimal generator of the process $t\mapsto\eta(t)$, defined for \emph{smooth cylinder functions} $f:\Omega\to\R$, is
\begin{align}
\label{infgen} Gf(\omega) = \sum_{e\in\cU} w(\omega(0)-\omega(e))
\big(f(\tau_e\omega)-f(\omega)\big) + \cD f(\omega)
\end{align}
where the (unbounded) linear operator
\begin{align}
\label{partial_op} \cD f(\omega) := \frac{\partial f}{\partial
\omega(0)}(\omega)
\end{align}
is well-defined for smooth cylinder functions.

The meaning of the various terms on the right-hand side of \eqref{infgen} is clear: the terms in the sum are due to the random shifts of the environment caused by the jumps of the random walker while the last term on the right-hand side is due to the deterministic linear growth of local time at the site actually occupied by the random walker.

Next, we define a probability measure on $\Omega$ which will turn out to be
stationary and ergodic for the Markov process $t\mapsto\eta(t)$. Let
\begin{align}
\label{R}
R:\R\to[0,\infty), \qquad R(u):=\int_0^ur(v)\,\d v.
\end{align}
$R$ is strictly convex and even. We denote by $\d\pi(\omega)$ \emph{the unique
centered Gibbs measure} (Markov field) on $\Omega$ defined by the conditional
specifications for $\Lambda\subset \Z^d$ finite:
\begin{align}
\label{specifications}
\d\pi\left(\omega_\Lambda\bigm|\omega_{\Z^d\setminus\Lambda}\right) =
Z_{\Lambda}^{-1} \exp\left\{- \frac12 \sum_{ \stackrel {x,y\in\Lambda}
{\abs{x-y}=1} } R(\omega(x)-\omega(y)) - \sum_{ \stackrel {x\in\Lambda,
y\in\Lambda^c} {\abs{x-y}=1} } R(\omega(x)-\omega(y))
\right\}\,\d\omega_\Lambda.
\end{align}
Note that the (translation invariant) Gibbs measure given by the specifications \eqref{specifications} exists only in three and more dimensions. For information about gradient measures of this type, see \cite{funaki_05}. The measure $\d\pi$ is invariant under the spatial shifts and the dynamical system
$(\Omega, \pi, \tau_z: z\in\Z^d)$ is \emph{ergodic}.

In the particular case when $r(u)=u$, $R(u)=u^2/2$, the measure $\d
\pi(\omega)$ is the distribution of the massless free Gaussian field on $\Z^d$,
$d\ge3$ with expectations and covariances
\begin{align}
\label{mfgfcovar} \int_{\Omega} \omega(x)\,\d\pi(\omega)=0, \qquad
\int_{\Omega} \omega(x)\omega(y)\,\d\pi(\omega)=(-\Delta)^{-1}_{x,y}=:C(y-x)
\end{align}
where $\Delta$ is the lattice Laplacian. We will refer to this special setup as
\emph{the Gaussian case}.

We are ready now to formulate the results regarding the lattice model.

\begin{proposition}
\label{prop:statinarity+ergodicity}
The probability measure $\d\pi(\omega)$ is stationary and ergodic for the Markov process $t\mapsto\eta(t)\in\Omega$.
\end{proposition}

The law of large numbers for the displacement of the random walker drops out for free:

\begin{corollary}
\label{cor:lln}
For $\pi$-almost all initial profiles $\ell(0,\cdot)$, almost surely
\begin{align}
\label{lln}
\lim_{t\to\infty}\frac{X(t)}{t}=0.
\end{align}
\end{corollary}

However, the main results refer to the diffusive scaling limit of the displacement.

\begin{theorem}
\label{thm:diffusive_limit}
\newcounter{szaml1}
\begin{list}
{(\arabic{szaml1})}{\usecounter{szaml1}\setlength{\leftmargin}{1em}}

\item
If conditions \eqref{ellipticity}, \eqref{convexity}, \eqref{s_small} and \eqref{r_entire} hold for the rate function, then
\begin{align}
\label{diffusive_bounds}
0<\gamma
\le
\inf_{{\abs{e}=1}} \varliminf_{t\to\infty} t^{-1} \expect{(e\cdot X(t))^2}
\le
\sup_{{\abs{e}=1}} \varlimsup_{t\to\infty} t^{-1} \expect{(e\cdot X(t))^2}
<\infty.
\end{align}

\item
Assume that
\begin{align}
\label{polynomial}
r(u)=u,\qquad s(u)=s_4u^4+s_2u^2+s_0,
\end{align}
and we also make the technical assumption that ${s_4}/{\gamma}$ is sufficiently
small. Then the matrix of asymptotic covariances
\begin{align}
\label{asymptotic_covariances}
\sigma^2_{kl}:=\lim_{t\to\infty}t^{-1}\expect{X_k(t)X_l(t)}
\end{align}
exists and it is non-degenerate. The finite dimensional distributions of the rescaled displacement process
\begin{align}
\label{rescaled_displacement}
X_N(t) := N^{-1/2}X(Nt)
\end{align}
converge to those of a $d$ dimensional Brownian motion with covariance matrix
$\sigma^2$.

\end{list}
\end{theorem}

\begin{remarks}
\newcounter{szaml2}
\begin{list}
{(\arabic{szaml2})}{\usecounter{szaml2}\setlength{\leftmargin}{1em}}
\item We do not strive to obtain optimal constants in our conditions. The upper
bound imposed on the ratio $s_4/\gamma$ which emerges from the computations in
the proof in Section \ref{s:check_gsc} is rather restrictive but far from
optimal.

\item It is clear that, in dimensions $d\ge3$, other stationary distributions of the
process $t\mapsto\eta(t)$ must exist. In particular, due to transience of the
process $t\mapsto X(t)$, the stationary measure (presumably) reached from
starting with ``empty'' initial conditions $\ell(0,x)\equiv0$ certainly differs
from our $\d\pi$. Our methods and results are valid for the particular
stationary distribution $\d\pi$.

\end{list}

\end{remarks}

\subsection{SRBP} \label{ss:setup_SRBP}

The environment profile appearing on the right-hand side of \eqref{Brpolydiff}, \eqref{Brpolydiff2}, as seen in a moving coordinate frame tied to the current position of the process is $t\mapsto\eta(t,\cdot)$:
\begin{align}
\label{envir}
\eta(t,x)
:=&\,\,
\eta(0,X(t)+x)+
\int_0^t V(X(t)+x-X(u))\,\d u
\\
\notag =&\,\, \big(V*\ell (t,\cdot)\big)(X(t)+x).
\end{align}
Here, $\eta(0,x):=\big(V*l(0,\cdot)\big)(x)$ is the initial condition for the
process. $t\mapsto\eta(t,\cdot)$ is a Markov process with continuous sample
paths in the Fr\'echet space
\begin{equation}
\label{OmegaSRBP} \Omega := \big\{\omega\in C^{\infty}(\R^d\to\R)\,:\,
\norm{\omega}_{m,r}<\infty \big\}
\end{equation}
where $\norm{\omega}_{m,r}$ are the seminorms
\begin{equation}
\label{seminorms}
\norm{\omega}_{m,r} := \sup_{x\in\R^d} \,
\big(1+\abs{x}\big)^{-1/r} \, \abs{\partial^{\abs{m}}_{m_1,\dots,m_d}\omega(x)}
\end{equation}
defined for the multiindices $m=(m_1,\dots,m_d)$, $m_j\ge0$; and $r\ge1$. The
group of spatial shifts
\begin{align}
\label{shifts_cont} \tau_z:\Omega\to\Omega, \qquad \tau_z\omega(x)=\omega(x+z),
\qquad z\in\R^d
\end{align}
acts naturally on the space $\Omega$.

In order to write down the infinitesimal generator of the process
$t\mapsto\eta(t)\in\Omega$, we define the following linear operators acting on \emph{smooth cylinder functions}
\begin{align}
\label{grad}
\nabla_j f(\omega)
&:=
\lim_{\vareps\to0}\vareps^{-1}\big(f(\tau_{\vareps e_j}\omega)-f(\omega)\big),
\\
\label{lap}
\Delta f(\omega)
&:=
\sum_{j=1}^d \nabla_j^2 f(\omega),
\\
\label{dirder} \cD f(\omega) &:=
\lim_{\vareps\to0}\vareps^{-1}\big(f(\omega+\vareps V)-f(\omega)\big).
\end{align}
It is easy to see that these linear operators are indeed well defined for smooth cylinder functions $f:\Omega\to \R$ of the form
\begin{align}
\label{cylf} f(\omega)=F(\omega(x_1),\dots,\omega(x_n)), \qquad x_j\in\R^d,
\qquad F\in C^{\infty}(\R^d\mapsto\R).
\end{align}

One finds that the infinitesimal generator of the process $t\mapsto\eta(t)$,
acting on smooth cylinder functions, is
\begin{align}
\label{infgen_p} Gf(\omega) = \frac12\Delta f(\omega) - \sum_{l=1}^d
\partial_l\omega(0) \nabla_l f(\omega) + \cD f(\omega).
\end{align}
The first two terms come from the infinitesimal shift of profile due to displacement of the particle, the last one from the infinitesimal change of profile due to increase of local time.

We denote by $\pi$ the Gaussian probability measure on $\Omega$ with expectation and covariances
\begin{equation}
\label{cov} \int_\Omega\omega(x) \d\pi(\omega)=0, \qquad
\int_\Omega\omega(x)\omega(y) \d\pi(\omega)= g*V(x-y) =:C(x-y)
\end{equation}
where
\begin{equation}
\label{greenf}
g:\R^d\to\R, \qquad
g(x):= \abs{x}^{2-d}
\end{equation}
is the Green function of the Laplacian in $\R^d$. (Mind that throughout this paper $d\ge3$.)
This is the \emph{massless free Gaussian field} whose ultraviolet singularity is smeared out by convolution with the unique positive definite solution of $U*U=V$. ($U$ is itself a smooth and rapidly decaying approximate identity.) The Fourier transform of the covariance is
\begin{equation}
\label{covarFour}
\wh C(p)=\abs{p}^{-2} \wh V(p).
\end{equation}
It will turn out that $\pi$ is a stationary (in time) and ergodic distribution for the Markov process $\eta(\cdot)$. All results will be meant for the process being in this stationary regime.

The results for the Brownian polymer model are the analogues to those for the
TSAW:

\begin{proposition}
\label{prop:stat_erg_poly} The Gaussian probability measure $\pi(\d\omega)$ on
$\Omega$ with mean $0$ and covariances \eqref{cov} is time-invariant and
ergodic for the $\Omega$-valued Markov process $t\mapsto\eta(t)$.
\end{proposition}

\begin{corollary}
\label{cor:lln_poly}
For $\pi$-almost all initial profiles $\eta(0,\cdot)$,
\begin{equation}
\label{lln2}
\lim_{t\to\infty}\frac{X(t)}{t}=0
\quad
\mathrm{a.s.}
\end{equation}
\end{corollary}

\begin{theorem}
\label{thm:clt_poly}
In dimensions $d\ge3$, the following hold:
\newcounter{szaml3}
\begin{list}{(\arabic{szaml3})}{\usecounter{szaml3}\setlength{\leftmargin}{1em}}
\item
The limiting variance
\begin{equation}
\label{variance2}
\sigma^2:=d^{-1}\lim_{t\to\infty}t^{-1}\expect{\abs{X(t)}^2}
\end{equation}
exists and
\begin{equation}
\label{bounds}
1 \le \sigma^2 \le 1+\rho^2
\end{equation}
where
\begin{equation}
\label{sumcond}
\rho^2:= d^{-1}\int_{\R^d} \abs{p}^{-2}\wh V(p)\,\d p<\infty.
\end{equation}
\item
The finite dimensional marginal distributions of the diffusively rescaled process
\begin{equation}
\label{rescaled2}
X_N(t)
:=
\frac{X(Nt)}{\sigma \sqrt N}
\end{equation}
converge to those of a standard $d$ dimensional Brownian motion. The
convergence is meant in probability with respect to the starting state
$\eta(0)$ sampled according to $\d\pi$.
\end{list}
\end{theorem}

The remarks formulated after Theorem \ref{thm:diffusive_limit} are equally valid for the SRBP model, too.

\section{Spaces and operators}
\label{s:spaces_and_operators}

\subsection{TSAW, general case}
\label{ss:spaces_and_operators_TSAW_general}

We put ourselves in the Hilbert space $\cH:=\cL^2(\Omega, \pi)$ and define some linear operators. The following shift and difference operators will be used:
\begin{align}
\label{diffops}
T_e f(\omega):=f(\tau_e\omega),
\qquad
\nabla_e:=T_{e}-I,
\qquad
\Delta:=\sum_{e\in\cU}\nabla_e=-\frac12\sum_{e\in\cU}\nabla_{e}\nabla_{-e}.
\end{align}
Their adjoints are
\begin{align}
\label{diffopsadj}
T_e^*=T_{-e},
\qquad
\nabla_e^*=\nabla_{-e},
\qquad
\Delta^*=\Delta.
\end{align}
Occasionally, we shall also use the notation $\nabla_l:=\nabla_{e_l}$.

We also define the multiplication operators
\begin{align}
\label{multops1}
M_e f(\omega)
&:=
s(\omega(0)-\omega(e)) f(\omega),
\\
\label{multops2}
N_e f(\omega)
&:=
r(\omega(0)-\omega(e)) f(\omega),
\qquad
N
:=
\sum_{e\in\cU} N_e.
\end{align}
These are unbounded self-adjoint operators. The following commutation relations are straightforward:
\begin{align}
\label{cr1}
M_eT_e-T_eM_{-e}
=0=
N_eT_e+T_eN_{-e}.
\end{align}
The (unbounded) differential operator $\cD$ is defined in \eqref{partial_op} on
the dense subspace of smooth cylinder functions and it is extended by graph
closure. Integration by parts on $(\Omega,\pi)$ yields
\begin{align}
\label{partial_adjoint} \cD+\cD^*=2N.
\end{align}

Next, we express the infinitesimal generator \eqref{infgen} of the semigroup of
the Markov process $t\mapsto\eta(t)$ acting on $\cL^2(\Omega,\pi)$. Denote
\begin{align}
S:=-\frac12(G+G^*), \qquad A:=\frac12(G-G^*)\label{symantisym}
\end{align}
the self-adjoint, respectively, skew self-adjoint parts of the infinitesimal generator. Using \eqref{cr1} and \eqref{partial_adjoint}, we readily obtain
\begin{align}
\label{symm_gen} S &= -\gamma \Delta + S_1,
\\
\label{symm_gen_1}
S_1
&=
-
\sum_{e\in\cU} M_e\nabla_e
=
\frac12 \sum_{e\in\cU} \nabla_{-e} M_e\nabla_e,
\\
\label{skew_symm_gen} A &= \phantom{-} \sum_{e\in\cU} N_e T_{e} +\big(\cD
-N\big).
\end{align}
Note that both $-\gamma \Delta$ and $S_1$ are \emph{positive operators}.
Actually, $\gamma \Delta$ is the infinitesimal generator of the process of
``scenery seen by the random walker'' (in the so-called RW in random scenery)
and $-S_1$ is the infinitesimal generator of ``environment seen by the random
walker in a symmetric RWRE''.

It is also worth noting that, defining the unitary involution
\begin{align}
\label{J-op}
Jf(\omega):=f(-\omega),
\end{align}
we get
\begin{align}
\label{yaglom_TSAW}
JSJ=S,
\quad
JAJ=-A,
\qquad
JGJ=G^*.
\end{align}
Stationarity drops out: indeed, $G^*\one=0$. Actually, \eqref{yaglom_TSAW} means slightly more than stationarity: the time-reversed and flipped process
\begin{align}
\label{rev-flip} t\mapsto\wt\eta(t):=-\eta(-t)
\end{align}
is equal in law to the process $t\mapsto\eta(t)$. This time reversal symmetry
is called \emph{Yaglom reversibility} and it appears in many models with
physical symmetries. See e.g.\ \cite{dobrushin_suhov_fritz_88},
\cite{yaglom_47}, \cite{yaglom_49}.

Ergodicity is also straightforward: for the Dirichlet form of the process $t\mapsto\eta(t)$ we have
\begin{align}
\label{erg} (f,-Gf)=(f,Sf)\ge\gamma(f,-\Delta f) =
\frac12\sum_{e\in3\cU}\norm{\nabla_e f}^2,
\end{align}
and hence, $Gf=0$ implies $\nabla_e f=0$, $e\in\cU$, which, in turn, by
ergodicity of the shifts on $(\Omega,\pi)$, implies $f=\text{const.}\one$.
Hence, Proposition \ref{prop:statinarity+ergodicity} and Corollary
\ref{cor:lln}.

\subsection{TSAW, the Gaussian case}
\label{ss:spaces_and_operators_TSAW_Gaussian}

\subsubsection{Spaces}

In the case where $r(u)=u$, the stationary measure defined by
\eqref{specifications} is Gaussian, and we can build up the Gaussian Hilbert space $\cH=\cL^2(\Omega,\pi)$ and its unitary equivalent representations as Fock spaces in the usual way.

We use the following convention for normalization of  Fourier transform
\begin{equation}
\label{lattice_FourierTransform} \wh u(p) = \sum_{x\in\Z^d}e^{i p\cdot x}u(x),
\qquad u(x) = (2\pi)^{-d}\int_{(-\pi,\pi]^d} e^{-i p\cdot x} \wh u(p)\,\d p,
\end{equation}
and the shorthand notation
\begin{align}
\label{notation1}
&
\vx=(x_1,\dots,x_n)\in\Z^{dn},
&&
x_m=(x_{m1},\dots,x_{md})\in\Z^d,
\\
\label{notation3}
&
\vp=(p_1,\dots,p_n)\in(-\pi,\pi]^{dn},
&&
p_m=(p_{m1},\dots,p_{md})\in(-\pi,\pi]^d,
\end{align}
$m=1,\dots,n$.

We denote by $\cS_n$, respectively, $\wh \cS_n$, the Schwartz space of
symmetric test functions of $n$ variables on $\Z^d$, respectively, on
$(-\pi,\pi]^d$:
\begin{align}
\label{SndefTSAW} \cS_n:&=\{u:\Z^{dn}\to\C:u\ \text{of rapid decay},
u(\varpi\vx)=u(\vx),\varpi\in\Perm(n)\},
\\
\label{hSndefTSAW} \wh \cS_n:&=\{\wh u:[-\pi,\pi]^{dn}\to\C:\wh u\in
C^{\infty}, \wh u(\varpi\vp)=\wh u(\vp), \varpi\in\Perm(n)\}.
\end{align}
In the preceding formulas $\Perm(n)$ denotes the symmetric group of permutations acting on the $n$ indices.

As noted before, in the case of $r(u)=u$, the random variables  $\big(\omega(x):x\in\Z^d\big)$ form the \emph{massless free Gaussian field} on $\Z^d$ with expectation and covariances given in \eqref{mfgfcovar}. The Fourier transform of the covariances is
\begin{align}
\wh C(p)=\wh D(p)^{-1}
\end{align}
where $\wh D(p)$ is the Fourier transform of the lattice Laplacian:
\begin{align}
\label{lat_Lap_Four}
\wh D(p):=
\sum_{l=1}^d \big(1-\cos(p_l)\big).
\end{align}

We endow the spaces $\cS_n$, respectively, $\wh \cS_n$
with the following scalar products
\begin{align}
\label{KnscprodTSAW}
&
\langle u,v\rangle:=
\sum_{\vx\in\Z^{dn}}\sum_{\vy\in\Z^{dn}} C(\vx-\vy)\overline{u(\vx)}v(\vy),
\\
\label{hKnscprodTSAW} & \langle \wh u,\wh v\rangle:= \int_{[-\pi,\pi]^{dn}}
\wh C(\vp) \overline{\wh u(\vp)} \wh v(\vp) \,\d\vp
\end{align}
where
\begin{equation}
\label{covkernels}
C(\vx-\vy):=\prod_{m=1}^n C(x_m-y_m),
\qquad
\wh C(\vp):=\prod_{m=1}^n \wh C(p_m).
\end{equation}
Let $\cK_n$ and $\wh \cK_n$ be the closures of $\cS_n$, respectively, $\wh
\cS_n$ with respect to the Euclidean norms defined by these inner products. The
Fourier transform \eqref{lattice_FourierTransform} realizes an isometric
isomorphism between the Hilbert spaces $\cK_n$ and $\wh \cK_n$.

These Hilbert spaces are actually the symmetrized $n$-fold tensor products
\begin{equation}
\label{KnhKn} \cK_n:=\mathrm{symm}\big(\cK_1^{\otimes n}\big), \qquad
\wh\cK_n:=\mathrm{symm}\big(\wh\cK_1^{\otimes n}\big).
\end{equation}
Finally, the full Fock spaces are
\begin{equation}
\label{Fockspaces}
\cK:=\overline{\oplus_{n=0}^\infty \cK_n}, \qquad
\wh\cK:=\overline{\oplus_{n=0}^\infty \wh\cK_n}.
\end{equation}

The Hilbert space of our true interest is $\cH=\cL^2(\Omega,\pi)$. This is itself a graded Gaussian Hilbert space
\begin{equation}
\label{Hgraded}
\cH=\overline{\oplus_{n=0}^\infty \cH_n}
\end{equation}
where the subspaces $\cH_n$ are isometrically isomorphic with the subspaces $\cK_n$ of $\cK$ through the identification
\begin{equation}
\label{HKisometryTSAW} \phi_n: \cK_n\to\cH_n, \quad
\phi_n(u):=\frac{1}{\sqrt{n!}}\sum_{\vx\in\Z^{dn}}
u(\vx)\wick{\omega(x_1)\dots\omega(x_n)}.
\end{equation}
Here and in the rest of this paper, we denote by \wick{X_1\dots X_n} the Wick product of the jointly Gaussian random variables $(X_1,\dots,X_n)$.

As the graded Hilbert spaces
\begin{equation}
\cH:=\overline{\oplus_{n=0}^\infty \cH_n}, \quad
\cK:=\overline{\oplus_{n=0}^\infty \cK_n}, \quad \wh
\cK:=\overline{\oplus_{n=0}^\infty \wh\cK_n}\label{3fullspaces}
\end{equation}
are isometrically isomorphic in a natural way, we shall move freely between the various representations.

\subsubsection{Operators}

First, we give the action of the operators $\nabla_e$, $\Delta$, etc.\
introduced in Subsection \ref{ss:spaces_and_operators_TSAW_general} on the
spaces $\cH_n$, $\cK_n$ and $\wh\cK_n$. The point is that we are interested
primarily in their action on the space $\cL^2(\Omega,\pi) =
\overline{\oplus_{n=0}^\infty\cH_n}$, but explicit computations in later
sections are handy in the unitary equivalent representations over the space
$\wh\cK=\overline{\oplus_{n=0}^\infty\wh\cK_n}$. The action of various
operators over $\cH_n$ will be given in terms of the Wick monomials
$\wick{\omega(x_1)\dots\omega(x_n)}$ and it is understood that the operators
are extended by linearity and graph closure.

\begin{itemize}

\item
The operators $\nabla_e$, $e\in\cU$ map $\cH_n\to\cH_n$, $\cK_n\to\cK_n$,
$\wh\cK_n\to\wh\cK_n$, in turn, as follows:
\begin{align}
\label{nablaonHn}
&
\nabla_e\wick{\omega(x_1)\dots\omega(x_n)}
=
\wick{\omega(x_1+e)\dots\omega(x_n+e)}-\wick{\omega(x_1)\dots\omega(x_n)},
\\
\label{nablaonKn}
&
\nabla_e u(\vx)
=
u(x_1-e,\dots,x_n-e)-u(x_1,\dots,x_n),
\\
\label{nablaonhKn}
&
\nabla_e \wh u(\vp)
=
\big(
\exp\left(i\textstyle\sum_{m=1}^n p_{m}\cdot e \right) -1 \big) \wh u(\vp).
\end{align}

\item
The operator $\Delta$ maps $\cH_n\to\cH_n$, $\cK_n\to\cK_n$,  $\wh\cK_n\to\wh\cK_n$, in turn, as follows:
\begin{align}
\label{DeltaonHn}
&
\Delta \wick{\omega(x_1)\dots\omega(x_n)}
=
\sum_{e\in\cU}
\wick{\omega(x_1+e),\dots,\omega(x_n+e)}
-2d \wick{\omega(x_1)\dots\omega(x_n)},
\\
\label{DeltaonKn}
&
\Delta u(\vx)
=
\sum_{e\in\cU}
u(x_1+e,\dots,x_n+e)-2d u(\vx),
\\
\label{DeltaonhKn}
&
\Delta \wh u(\vp) =
-2\wh D \left(\textstyle\sum_{m=1}^n p_m\right) \wh u(\vp).
\end{align}

\item
The operators $\abs{\Delta}^{-1/2}\nabla_e$ map $\cH_n\to\cH_n$, $\cK_n\to\cK_n$,  $\wh\cK_n\to\wh\cK_n$. There is no explicit expression for the first two. The action $\wh\cK_n\to\wh\cK_n$ is as follows:
\begin{equation}
\label{nonameopsonhKn} \abs{\Delta}^{-1/2}\nabla_e \wh u(\vp) = \frac{\exp
\left(i\textstyle\sum_{m=1}^n p_{m}\cdot e \right) - 1} {\sqrt{2\wh D
\left(\textstyle\sum_{m=1}^n p_m\right)}} \wh u(\vp).
\end{equation}
These are \emph{bounded} operators with norm
\begin{equation}
\label{nonameopnorm}
\norm{ \abs{\Delta}^{-1/2}\nabla_e } =1.
\end{equation}

\item
The creation operators $a^*_e$, $e\in\cU$ map $\cH_n\to\cH_{n+1}$,
$\cK_n\to\cK_{n+1}$,  $\wh\cK_n\to\wh\cK_{n+1}$, in turn, as follows:
\begin{align}
\label{astaronHn}
&
a^*_e\wick{\omega(x_1)\dots\omega(x_n)} =
\wick{(\omega(0)-\omega(e))\omega(x_1)\dots\omega(x_n)},
\\
\label{astaronKn}
&
a^*_e u(x_1,\dots,x_{n+1})
=
\frac{1}{\sqrt{n+1}}
\sum_{m=1}^{n+1}
\big(\delta_{x_m,0}-\delta_{x_m,e}\big)
u(x_1,\dots, \cancel{x_{m}},\dots, x_{n+1}) ,
\\
\label{astaronhKn}
&
a^*_e \wh u(p_1,\dots,p_{n+1})
=
\frac{1}{\sqrt{n+1}}
\sum_{m=1}^{n+1}
\big(e^{i p_m\cdot e} -1 \big)
\wh u(p_1,\dots, \cancel{p_{m}},\dots, p_{n+1}).
\end{align}
The creation operators $a^*_e$ restricted to the subspaces $\cH_n$, $\cK_n$,
respectively, $\wh\cK_n$ are bounded with operator norm
\begin{equation}
\label{astaropnorm}
\norm{a^*_e\upharpoonright_{\cH_n}} =
\norm{a^*_e\upharpoonright_{\cK_n}} = \norm{a^*_e\upharpoonright_{\wh\cK_n}}
= \sqrt{2C(0)-C(e)-C(-e)} \sqrt{n+1}.
\end{equation}

\item
The annihilation operators $a_e$, $e\in\cU$ map $\cH_n\to\cH_{n-1}$,
$\cK_n\to\cK_{n-1}$,  $\wh\cK_n\to\wh\cK_{n-1}$, in turn, as follows:
\begin{align}
\label{aonHn} & a_e \wick{\omega(x_1)\dots\omega(x_n)} = \sum_{m=1}^n
\big(C(x_m+e)-C(x_m)\big)
\wick{\omega(x_1)\dots\cancel{\omega(x_{m})}\dots\omega(x_n)},
\\
\label{aonKn} & a_e u(x_1,\dots,x_{n-1}) = \sqrt{n} \sum_{z\in\Z^d}
\big(C(z+e)-C(z)\big) u(x_1,\dots,x_{n-1},z),
\\
\label{aonhKn} & a_e \wh u(p_1,\dots,p_{n-1}) = \sqrt{n} (2\pi)^{-d}
\int_{[-\pi,\pi]^d} \big(e^{-i q\cdot e} -1\big) \wh C(q)  \wh
u(p_1,\dots,p_{n-1},q)\,\d q.
\end{align}
The annihilation operators $a_e$ restricted to the subspaces $\cH_n$, $\cK_n$, respectively, $\wh\cK_n$ are bounded with operator norm
\begin{equation}
\label{aopnorm}
\norm{a_e\upharpoonright_{\cH_n}}=
\norm{a_e\upharpoonright_{\cK_n}}=
\norm{a_e\upharpoonright_{\wh\cK_n}} =
\sqrt{2C(0)-C(e)-C(-e)} \sqrt{n}.
\end{equation}
As the notation $a^*_e$ and $a_e$ suggests, these operators are adjoints of each other.
\end{itemize}

In order to express the infinitesimal generator in the Gaussian case, two more
observations are needed. Both follow from standard facts in the context of
Gaussian Hilbert spaces or Malliavin calculus.  First, the operator of
multiplication by $\omega(0)-\omega(e)$, acting on $\cH$, is $a^*_e+a_e$.
Hence, the multiplication operators $M_e$ and $N_e$ defined in \eqref{multops1}
and \eqref{multops2}, in the Gaussian case, are
\begin{align}
\label{gaussianmultops}
N_e=a^*_e+a_e,
\qquad
M_e= s(a^*_e+a_e).
\end{align}
Second, from the formula of \emph{directional derivative} in $\cH$, it follows that
\begin{align}
\label{gaussiandirder}
\cD=\sum_{e\in\cU} a_e.
\end{align}
The identity \eqref{gaussiandirder} is checked directly on Wick polynomials and extends by linearity.

Using these identities, after simple manipulations, we obtain
\begin{align}
\label{opsgrading}
S_1
&=
\frac12 \sum_{e\in\cU} \nabla_{-e} s(a_e^*+a_l) \nabla_e,
\\
A &= \sum_{e\in\cU} \nabla_{-e} a_e - \sum_{e\in\cU}  a^*_e\nabla_{-e} =: A_- -
A_+.\label{opagrading}
\end{align}
Note that
\begin{align}
\label{opgrading} \displaystyle A_\pm:\cH_n\to\cH_{n\pm1}, \qquad S_1:
\cH_{n}\to\oplus_{j=-q}^q\cH_{n+2j}
\end{align}
where $2q$ is the degree of the even polynomial $s(u)$.

\subsection{SRBP, Gaussian}
\label{ss:spaces_and_operators_SRBP_Gaussian}

We use the convention of unitary Fourier transform
\begin{equation}
\label{FourierTransform} \wh u(p) := (2\pi)^{-d/2}\int_{\R^d}e^{i p\cdot
x}u(x)\,\d x, \qquad u(x) := (2\pi)^{-d/2}\int_{\R^d}e^{-i p\cdot x}\wh
u(p)\,\d p
\end{equation}
and the shorthand notation corresponding to \eqref{notation1} and
\eqref{notation3} with $\vx\in\R^{dn}$ and $\vp\in\R^{dn}$.

Now, we denote by $\cS_n$, respectively, $\wh \cS_n$ the \emph{symmetric}
Schwartz space of test functions and their Fourier transforms:
\begin{align}
\label{SndefSRBP} \cS_n :&=\{u:\R^{dn}\to\C:u \text{ smooth, of rapid
decay},u(\varpi\vx)=u(\vx),\varpi\in\Perm(n)\},
\\
\label{hSndefSRBP} \wh \cS_n :&=\{\wh u:\R^{dn}\to\C:\wh u \text{ smooth, of
rapid decay},\wh u(\varpi\vp)=\wh u(\vp),\varpi\in\Perm(n)\}.
\end{align}
The spaces $\cS_n$, respectively, $\wh \cS_n$ are endowed with the following scalar products
\begin{align}
\label{KnscprodSRBP}
&
\langle u,v\rangle:=
\int_{\R^{dn}}\int_{\R^{dn}}
C(\vx-\vy)\overline{u(\vx)}v(\vy) \,\d\vx\d\vy,
\\
\label{hKnscprodSRBP}
&
\langle \wh u,\wh v\rangle:=
\int_{\R^{dn}}
 \wh C(\vp) \overline{\wh u(\vp)}\wh v(\vp) \,\d\vp
\end{align}
where the kernels $C(\vx-\vy)$ and $\wh C(\vp)$ are expressed with formulas similar to \eqref{covkernels}.

Next, we can build up the three isometrically isomorphic representations of the graded Hilbert space $\cL^2(\Omega,\pi)$: $\cH$, $\cK$ and $\wh\cK$ as in
\eqref{KnhKn}--\eqref{3fullspaces}. The identification between $\cK_n$ and
$\cH_n$ is
\begin{equation}
\label{HKisometrySRBP} \phi_n: \cK_n\to\cH_n, \quad
\phi_n(u):=\frac{1}{\sqrt{n!}}\int_{\R^{dn}}
u(\vx)\wick{\omega(x_1)\dots\omega(x_n)}\,\d\vx,
\end{equation}
the one between $\cK_n$ and $\wh\cK_n$ is the Fourier transformation
\eqref{FourierTransform}.

We enumerate the operators introduced in Subsection \ref{ss:setup_SRBP}. The
operators which are denoted by the same symbol in the two models behave quite
similarly with natural modifications: the discrete difference in the TSAW model
is replaced by differentiation in the SRBP case. In order to spare space, we
give their actions only on $\wh\cK_n$, which are the most useful formulae for
later calculations. For more details, see \cite{horvath_toth_veto_09}.
\begin{itemize}
\item
The operators $\nabla_l:\wh\cK_n\to\wh\cK_n$, $l=1,\dots,d$:
\begin{equation}
\nabla_l \wh u(\vp) = i \big(\sum_{m=1}^n p_{ml}\big) \wh u(\vp).
\end{equation}
Note that these are actually unbounded, closed, skew self-adjoint operators.  They are the \emph{second quantized gradients}.

\item
The operator $\Delta:\wh\cK_n\to\wh\cK_n$:
\begin{equation}
\Delta \wh u(\vp) = -\abs{\sum_{m=1}^n p_m}^2 \wh u(\vp).
\end{equation}
The operator $\Delta$ is unbounded, densely defined, self-adjoint and positive. Note that $\Delta$ is \emph{not} the second quantized Laplacian.

\item
The operators $\abs{\Delta}^{-1/2}\nabla_l:\wh\cK_n\to\wh\cK_n$, $l=1,\dots,d$:
\begin{equation}
\abs{\Delta}^{-1/2}\nabla_l \wh u(\vp) = \frac{i\sum_{m=1}^n p_{ml}
}{\abs{\sum_{m=1}^n p_m}} \wh u(\vp).
\end{equation}
These are \emph{bounded} skew self-adjoint operators with operator norm
\begin{equation}
\norm{ \abs{\Delta}^{-1/2}\nabla_l } =1,\label{nonameopnormSRBP}
\end{equation}
however $|\Delta|^{-1/2}$ itself is an unbounded, densely defined, self-adjoint
and positive operator.

\item
The creation operators $a^*_l:\wh\cK_n\to\wh\cK_{n+1}$, $l=1,\dots,d$:
\begin{equation}
a^*_l\wh u(p_1,\dots,p_{n+1}) = \frac{1}{\sqrt{n+1}} \sum_{m=1}^{n+1} ip_{ml}
\wh u(p_1,\dots, \cancel{p_{m}},\dots, p_{n+1}).
\end{equation}
The creation operators $a_l^*$, restricted to the subspace $\wh\cK_n$ are
bounded with operator norm
\begin{equation}
\norm{ a^*_l\upharpoonright_{\wh\cK_n} \!\!\!\phantom{\Big|}} = \sqrt{-2
\partial^2_lC(0)} \sqrt{n+1}.
\end{equation}

\item
The annihilation operators $a_l:\wh\cK_n\to\wh\cK_{n-1}$, $l=1,\dots,d$:
\begin{equation}
a_l\wh u(p_1,\dots,p_{n-1}) = \sqrt{n} \int_{\R^d} \wh
u(p_1,\dots,p_{n-1},q)iq_l\wh C(q) \,\d q.
\end{equation}
The annihilation operators $a_l$ restricted to the subspace $\wh\cK_n$ are
bounded with operator norm
\begin{equation}
\norm{ a_l\upharpoonright_{\wh\cK_n} \!\!\!\phantom{\Big|}} =
\sqrt{-2\partial_l^2 C(0)}\sqrt{n}.
\end{equation}
Furthermore, as the notation $a^*_l$ and $a_l$ suggests, these operators are adjoint of each other.
\end{itemize}

Since all computations are performed in the representation $\wh\cK$, we give a
common core for all the unbounded operators defined above -- and some others to
appear in future sections:
\begin{equation}
\label{coredefin}
\wh\cC:=\oplus_{n=0}^\infty \wh\cC_n,
\qquad
\wh\cC_n:=
\{\wh u \in \wh\cK_n: \sup_{\vp\in\R^{dn}}\abs{\wh u(\vp)} < \infty \}.
\end{equation}
Note that the operator $\abs{\Delta}^{-1/2}$ is defined on the dense subspace
$\wh\cC$ \emph{only for} $d\ge3$. Furthermore, in dimensions $d\ge3$, the
operators $\abs{\Delta}^{-1/2}\upharpoonright_{\wh\cK_n}$ defined on the dense
subspaces $\wh\cC_n$, are \emph{essentially self-adjoint}. This follows, e.g.\
from Propositions VIII.1, VIII.2 of \cite{reed_simon_vol1_80}.

Notice also that $\nabla$ is the infinitesimal generator of the \emph{unitary group of spatial translations} while $\Delta$ is the infinitesimal generator of the Markovian semigroup of \emph{diffusion in random scenery}
\begin{align}
\label{shiftgroup}
&
\exp\{z\nabla\}=T_z, \qquad
&&
T_zf(\omega):=f(\tau_z\omega),
\\
\label{drscesemigroup}
&
\exp\{t\Delta\}=Q_t, \qquad
&&
Q_tf(\omega):=\int\frac{\exp\{-z^2/(2t)\}}{\sqrt{2\pi t}}
f(\tau_z\omega)\,\d z.
\end{align}

Next, we express the infinitesimal generator \eqref{infgen_p} of the semigroup
of the process $t\mapsto\eta(t)$ acting on $\cL^2(\Omega, \pi)$. Using the
\emph{directional derivative formula} of Malliavin calculus and standard
commutation relations for the operators $\nabla_l, a^*_l, a_l$, $l=1,\dots,d$,
we get the expression
\begin{equation}
\label{infgen_poly_in_H}
G
:=
\frac12\Delta +
\sum_{l=1}^d \big(a^*_l \nabla_l + \nabla_l a_l\big).
\end{equation}
In order to spare space, we omit the more or less routine computations. For
details, we refer the reader to \cite{horvath_toth_veto_09} and
\cite{tarres_toth_valko_09}. This operator is well defined on Wick polynomials
of the field $\omega(x)$ and is extended by linearity and graph closure. One
can check that it satisfies the criteria of the Hille\,--\,Yoshida theorem (see
\cite{reed_simon_vol1_80}) and thus it is indeed the infinitesimal generator of
a Markovian semigroup. We omit these technical details.

The adjoint generator is
\begin{equation}
\label{adjinfgen_poly_in_H}
G^*
:=
\frac12\Delta -
\sum_{l=1}^d \big(a^*_l \nabla_l + \nabla_l a_l\big).
\end{equation}
Note that, due to the inner coherence of the model, the various terms  of the
infinitesimal generator \eqref{infgen_p} combine to give the tidy skew
self-adjoint part of the infinitesimal generators in \eqref{infgen_poly_in_H}
and \eqref{adjinfgen_poly_in_H}.

The symmetric (self-adjoint) and anti-symmetric (skew self-adjoint) parts of
the generator which were introduced in \eqref{symantisym} for the TSAW are now
\begin{equation}
\label{SA:SRBP}
S=-\frac12\Delta,\qquad A=\sum_{l=1}^d \big(a^*_l \nabla_l +
\nabla_l a_l\big) =: A_++A_-.
\end{equation}
It is a standard -- though not completely trivial -- exercise to check that in $d\ge3$ the
operators $S$ and $A$, a priori defined on the dense subspace $\wh\cC$, are
indeed essentially self-adjoint, respectively, essentially skew self-adjoint.

Note that
\begin{equation}
\label{grading}
S:\cH_n\to\cH_n,
\quad
A_+:\cH_n\to\cH_{n+1},
\quad
A_-:\cH_n\to\cH_{n-1},
\quad
A_{\mp}=-A_{\pm}^*,
\end{equation}
and
\begin{equation}
\label{van_H0_H1}
S\upharpoonright_{\cH_0}=0,
\qquad
A_+\upharpoonright_{\cH_0}=0,
\qquad
A_-\upharpoonright_{\cH_0\oplus\cH_1}=0.
\end{equation}

The proof of stationarity and ergodicity is very similar to the lattice case: Yaglom
reversibility \eqref{yaglom_TSAW} holds also here, which gives the stationarity
of the measure $\pi$. Ergodicity follows from the same argument as in
\eqref{erg}. This proves Proposition \ref{prop:stat_erg_poly} and Corollary
\ref{cor:lln_poly}.

\section{CLT for additive functionals of ergodic Markov processes, graded sector condition}
\label{s:KV}

In the present section, we recall the non-reversible version of the
Kipnis\,--\,Varadhan CLT for additive functionals of ergodic Markov processes
and present a slightly enhanced version of the \emph{graded sector condition}
of Sethuraman, Varadhan and Yau, \cite{sethuraman_varadhan_yau_00}.

Let $(\Omega, \cF, \pi)$ be a probability space: the state space of a
\emph{stationary and ergodic} Markov process $t\mapsto\eta(t)$. We put
ourselves in the Hilbert space $\cH:=\cL^2(\Omega, \pi)$. Denote the
\emph{infinitesimal generator} of the semigroup of the process by $G$, which is
a well-defined (possibly unbounded) closed linear operator on $\cH$. The
adjoint generator $G^*$ is the infinitesimal generator of the semigroup of the
reversed (also stationary and ergodic) process $\eta^*(t)=\eta(-t)$. It is
assumed that $G$ and $G^*$ have a \emph{common core of definition}
$\cC\subseteq\cH$. Let $f\in\cH$, such that $(f, \one) = \int_\Omega
f\,\d\pi=0$. We ask about CLT/invariance principle for
\begin{equation}
\label{rescaledintegral}
N^{-1/2}\int_0^{Nt} f(\eta(s))\,\d s
\end{equation}
as $N\to\infty$.

We denote the \emph{symmetric} and \emph{anti-symmetric} parts of the generators $G$, $G^*$, by
\begin{equation}
S:=-\frac12(G+G^*),
\qquad
A:=\frac12(G-G^*).
\end{equation}
These operators are also extended from $\cC$ by graph closure and it is assumed that they are well-defined self-adjoint, respectively, skew self-adjoint operators
\begin{equation}
S^*=S\ge0, \qquad A^*=-A.
\end{equation}
Note that $-S$ is itself the infinitesimal generator of a Markovian semigroup
on $\cL^2(\Omega,\pi)$, for which the probability measure $\pi$ is reversible
(not just stationary). We assume that $-S$ is itself ergodic:
\begin{equation}
\label{Sergodic}
\mathrm{Ker}(S)=\{c1\!\!1 : c\in\C\}.
\end{equation}

We denote by $R_\lambda\in\cB(\cH)$ the resolvent of the semigroup $s\mapsto e^{sG}$:
\begin{equation}
R_\lambda
:=
\int_0^\infty e^{-\lambda s} e^{sG}\d s
=
\big(\lambda I-G\big)^{-1}, \qquad \lambda>0,
\end{equation}
and, given $f\in\cH$ as above, we will use the notation
\begin{equation}
u_\lambda:=R_\lambda f.
\end{equation}

The following theorem yields the efficient martingale approximation of the additive functional \eqref{rescaledintegral}:

\begin{theorem*}[\bf KV]
\label{thm:kv}
With the notation and assumptions as before, if the following two limits hold in $\cH$:
\begin{align}
\label{conditionA}
&
\lim_{\lambda\to0}
\lambda^{1/2} u_\lambda=0,
\\
\label{conditionB}
&
\lim_{\lambda\to0} S^{1/2} u_\lambda=:v\in\cH,
\end{align}
then
\begin{equation}
\label{kv_variance}
\sigma^2:=2\lim_{\lambda\to0}(u_\lambda,f)\in[0,\infty),
\end{equation}
and there exists a zero mean, $\cL^2$-martingale $M(t)$ adapted to the filtration of the Markov process $\eta(t)$ with stationary and ergodic increments and variance
\begin{equation}
\expect{M(t)^2}=\sigma^2t
\end{equation}
such that
\begin{equation}
\label{kv_martappr}
\lim_{N\to\infty} N^{-1} \expect{\big(\int_0^N
f(\eta(s))\,\d s-M(N)\big)^2} =0.
\end{equation}
In particular, if $\sigma>0$, then the finite dimensional marginal distributions of the rescaled process $t\mapsto \sigma^{-1} N^{-1/2}\int_0^{Nt}f(\eta(s))\,\d s$ converge to those of a standard $1d$ Brownian motion.
\end{theorem*}

\begin{remarks}
\newcounter{szaml6}
\begin{list}
{(\arabic{szaml6})}{\usecounter{szaml6}\setlength{\leftmargin}{1em}}
\item
Conditions \eqref{conditionA} and \eqref{conditionB} of the theorem are jointly
equivalent to the following
\begin{equation}
\label{conditionC}
\lim_{\lambda,\lambda'\to0}(\lambda+\lambda')(u_\lambda,u_{\lambda'})=0.
\end{equation}
Indeed, straightforward computations yield:
\begin{equation}
\label{A+B=C}
(\lambda+\lambda')(u_\lambda,u_{\lambda'}) =
\norm{S^{1/2}(u_\lambda-u_{\lambda'})}^2 + \lambda \norm{u_\lambda}^2 +
\lambda' \norm{u_{\lambda'}}^2.
\end{equation}

\item
The theorem is a generalization to non-reversible setup of the
celebrated Kipnis\,--\,Varadhan theorem, \cite{kipnis_varadhan_86}. To the best of our knowledge, the non-reversible formulation, proved with resolvent rather
than spectral calculus, appears first -- in discrete-time Markov chain, rather than continuous-time Markov process setup and with condition \eqref{conditionC} -- in \cite{toth_86} where it was applied, with bare hand computations, to obtain CLT for a particular random walk in random environment. Its proof follows the original proof of the  Kipnis\,--\,Varadhan theorem with the difference that spectral calculus is to be replaced by resolvent calculus.

\item
In continuous-time Markov process setup, it was formulated in
\cite{varadhan_96} and applied to tagged particle motion in non-reversible zero
mean exclusion processes. In this paper, the \emph{(strong) sector condition}
was formulated, which, together with an $H_{-1}$-bound on the function
$f\in\cH$, provide sufficient condition for \eqref{conditionA} and
\eqref{conditionB} of Theorem KV to hold.

\item
In \cite{sethuraman_varadhan_yau_00}, the so-called \emph{graded sector
condition} is formulated and Theorem KV is applied to tagged particle diffusion
in general (non-zero mean) non-reversible exclusion processes in $d\ge3$. The
fundamental ideas related to the graded sector condition have their origin
partly in \cite{landim_yau_97}. In Theorem GSC below, we quote -- and slightly
enhance -- the formulation in \cite{olla_01} and
\cite{komorowski_landim_olla_09}.

\item
For a more complete list of applications of Theorem KV together with the strong
and graded sector conditions, see the surveys \cite{olla_01} and
\cite{komorowski_landim_olla_09}.

\end{list}
\end{remarks}

Checking conditions \eqref{conditionA} and \eqref{conditionB} (or, equivalently, condition \eqref{conditionC}) in particular applications is typically not easy. In the applications to RWRE in \cite{toth_86}, the conditions were checked by bare hand computations. In \cite{varadhan_96}, respectively, \cite{sethuraman_varadhan_yau_00}, the so-called \emph{sector condition}, respectively, the \emph{graded sector condition} were introduced and checked for the respective models.

We reformulate the graded sector condition from \cite{olla_01} and
\cite{komorowski_landim_olla_09} in a somewhat enhanced version. From abstract
functional analytic considerations, it follows that the next two conditions
jointly imply \eqref{conditionA} and \eqref{conditionB}:
\begin{gather}
\label{conditionH-1} f\in\text{Ran}(S^{1/2}),
\\
\label{conditionD}
\sup_{\lambda>0}
\norm{ S^{-1/2} Gu_\lambda} < \infty.
\end{gather}
Assume that the Hilbert space $\cH=\cL^2(\Omega, \pi)$ is graded
\begin{equation}
\label{grading2}
\cH=\overline{\oplus_{n=0}^\infty\cH_n},
\end{equation}
and the infinitesimal generator is consistent with this grading in the following sense:
\begin{align}
\label{Sgrading} S&=\sum_{n=0}^\infty\sum_{j=-r}^rS_{n,n+j},&
S_{n,n+j}&:\cH_n\to\cH_{n+j},& S_{n,n+j}^*&=S_{n+j,n},
\\
\label{Agrading} A&=\sum_{n=0}^\infty\sum_{j=-r}^rA_{n,n+j},&
A_{n,n+j}&:\cH_n\to\cH_{n+j},& A_{n,n+j}^*&=-A_{n+j,n}.
\end{align}
Here and in the sequel, the double sum $\sum_{n=0}^\infty\sum_{j=-r}^r \cdots$
is meant as\\ $\sum_{n=0}^\infty\sum_{j=-r}^r \ind{n+j\ge0} \cdots$.

\begin{theorem*}[\bf GSC]
\label{thm:gsc} Let the Hilbert space and the infinitesimal generator be graded
in the sense specified above. Assume that there exists an operator $D=D^*\geq0$
which acts diagonally on the grading of $\cH$:
\begin{equation}
D=\sum_{n=0}^\infty D_{n,n}, \qquad D_{n,n}: \cH_n\to\cH_n
\end{equation}
such that
\begin{equation}
\label{Dellipticity}
0\le D\leq S.
\end{equation}
Assume also that, with some $C<\infty$ and $2\le\kappa<\infty$, the following
bounds hold:
\begin{align}
\label{diagbound} \norm{D_{n,n}^{-1/2} (S_{n,n}+A_{n,n}) D_{n,n}^{-1/2}} &\leq
C n^{\kappa},
\\
\label{oddbound} \norm{ D_{n+j,n+j}^{-1/2} A_{n,n+j} D_{n,n}^{-1/2}} &\leq
\frac{n}{12 r^2 \kappa}+C,& j&=\pm1,\dots,\pm r,
\\
\label{evenbound} \norm{D_{n+j,n+j}^{-1/2} S_{n,n+j} D_{n,n}^{-1/2}} &\leq
\frac{n^2}{6 r^3 \kappa^2}+C,& j&=\pm1,\dots,\pm r,
\end{align}

Under these conditions on the operators, for any function $f\in \oplus_{n=0}^N \cH_n$, with some  $N<\infty$, if
\begin{equation}
\label{Dupperbound} D^{-1/2}f\in\cH,
\end{equation}
then \eqref{conditionH-1} and \eqref{conditionD} follow. As a consequence, the martingale approximation and CLT of Theorem KV hold.
\end{theorem*}

\begin{remark}
In the original formulation of the graded sector condition (see
\cite{sethuraman_varadhan_yau_00}, \cite{komorowski_landim_olla_09} and
\cite{olla_01}), the bound  imposed in \eqref{evenbound} on the symmetric part
of the generator was of the same form as that imposed in \eqref{oddbound} on
the skew-symmetric part. We can go up to the bound of order $n^2$ (rather than
of order $n$) in \eqref{evenbound} due to decoupling of the estimates of the
self-adjoint and skew self-adjoint parts. The proof follows the main lines of
the original one with one extra observation which allows the enhancement
mentioned above. We present a \emph{sketchy proof} -- for the connoisseur --
below.
\end{remark}

\begin{proof}
We present a sketchy proof following the main steps and notations used in
\cite{olla_01} or \cite{komorowski_landim_olla_09} and \emph{emphasizing only
that point where we gain slightly more in the upper bound imposed in
\eqref{evenbound}}. The expert should jump directly to comparing the bounds
\eqref{secondright1} and \eqref{thirdright1} below.

Let
\begin{equation}
f=\sum_{n=0}^N f_n,
\qquad
u_\lambda=\sum_{n=0}^\infty u_{\lambda n},
\qquad
f_n,u_{\lambda n}\in \cH_n.
\end{equation}
From \eqref{diagbound}, \eqref{oddbound} and \eqref{evenbound}, it easily
follows that
\begin{equation}
\label{kappaineq} \left\| S^{-1/2}Gu_\lambda\right\|^2\leq C \sum_n n^{2\kappa}
\left\| D^{1/2}u_{\lambda n}\right\|^2
\end{equation}
with some $C<\infty$. So it suffices to prove that the right-hand side of \eqref{kappaineq} is bounded, uniformly in $\lambda>0$.

Let
\begin{equation}
t(n): =
n_1^\kappa \ind{0\le n< n_1}
+
n^\kappa \ind{n_1\le n \le n_2}
+
n_2^\kappa \ind{n_2< n <\infty}
\end{equation}
with the values of $0< n_1<n_2<\infty$ to be fixed later, and define the bounded linear operator $T:\cH\to\cH$,
\begin{equation}
\label{tdefin}
T\upharpoonright_{\cH_n}=t(n)I\upharpoonright_{\cH_n}.
\end{equation}

We start with the identity
\begin{equation}
\label{abcde}
\lambda( u_\lambda, T u_\lambda)
+
( T u_\lambda, STu_\lambda)
=
( T u_\lambda, Tf)
-
( T u_\lambda, [A,T]u_\lambda).
+
( T u_\lambda, [S,T]u_\lambda)
\end{equation}
obtained from the resolvent equation by straightforward manipulations. We
point out here that separating the last two terms on the right-hand side rather
than handling them jointly as $( T u_\lambda, [T,G]u_\lambda)$ (as done in the
original proof) will allow for gain in the upper bound imposed in
\eqref{evenbound}.

Just as in the original proof, we get the bounds:
\begin{align}
\label{firstleft} \lambda ( T u_\lambda, T u_\lambda ) &= \lambda\sum_n
t(n)^2\norm{u_{\lambda n}}^2 \geq 0,
\\
\label{secondleft} ( T u_\lambda, STu_\lambda ) &\geq \sum_n
t(n)^2\norm{D^{1/2}u_{\lambda n}}^2,
\\
\label{firstright} ( T u_\lambda, Tf ) &\leq \frac14 \sum_n
t(n)^2\norm{D^{1/2}u_{\lambda n}}^2 + \sum_n t(n)^2\norm{D^{-1/2}f_n}^2.
\end{align}

Now, the last two terms on the right-hand side of \eqref{abcde} follow. The
second term (containing $A$) is treated just like in the original proof, the
third term (containing $S$) slightly differently.
\begin{align}
\label{secondright1}
(Tu_\lambda, [A,T]u_\lambda)
&
=
\frac12
(u_\lambda, (AT^2-T^2A)u_\lambda)
\\
\notag
&=
\frac12 \sum_n\sum_{j=-r}^{r}
\left(t(n)^2-t(n+j)^2\right)
(u_{\lambda (n+j)}, A_{n,n+j}u_{\lambda n})
\\
\notag & \le \frac12 \sum_n\sum_{j=-r}^{r} \abs{t(n)^2-t(n+j)^2}
\left(\frac{n}{12 r^2 \kappa }+C\right) \norm{D^{1/2} u_{\lambda n}}
\norm{D^{1/2} u_{\lambda (n+j)}},
\\[5pt]
\label{thirdright1}
(Tu_\lambda, [S,T]u_\lambda)
&
=
\frac12
(u_\lambda, (2TST-ST^2-T^2S)u_\lambda)
\\
\notag
&
=
-\frac12 \sum_n\sum_{j=-r}^{r}
\big(t(n)-t(n+j)\big)^2
(u_{\lambda (n+j)}, S_{n,n+j}u_{\lambda n})
\\
\notag
&
\le
\frac12 \sum_n\sum_{j=-r}^{r}
\big(t(n)-t(n+j)\big)^2
\left(\frac{n^2}{6 r^3 \kappa^2}+C\right)
\norm{D^{1/2} u_{\lambda n}} \norm{D^{1/2} u_{\lambda (n+j)}}.
\end{align}
Note the difference between the coefficients in the middle lines of
\eqref{secondright1}, respectively, \eqref{thirdright1}. Choosing $n_1$
sufficiently large, we get
\begin{align}
\sup_n\max_{-r\le j\le r} \frac{\abs{t(n)^2-t(n+j)^2}}{t(n)^2}
\left(\frac{n}{12 r^2 \kappa }+ C\right) &\le \frac1{2(2r+1)},
\\
\sup_n\max_{-r\le j\le r} \frac{\big(t(n)-t(n+j)\big)^2}{t(n)^2} \left(
\frac{n^2}{6 r^3 \kappa^2} + C \right) &\le \frac1{2(2r+1)}.
\end{align}
and hence, via another Schwarz,
\begin{align}
\label{secondandthirdright}
\abs{(Tu_\lambda, [A,T]Tu_\lambda)}
+
\abs{(Tu_\lambda, [S,T]Tu_\lambda)}
\le
\frac12
\sum_n t(n)^2\norm{D^{1/2}u_{\lambda n}}^2.
\end{align}
Putting \eqref{abcde}, \eqref{firstleft}, \eqref{secondleft},
\eqref{firstright} and \eqref{secondandthirdright} together, we obtain:
\begin{align}
\sum_n t(n)^2 \norm {D^{1/2}u_{\lambda n}}^2
\leq
4 \sum_n t(n)^2 \norm{D^{-1/2}f_n}^2
=
4 \sum_{n=0}^N t(n)^2 \norm{D^{-1/2}f_n}^2.
\end{align}
Finally, letting $n_2\to\infty$, we get indeed \eqref{conditionD} via \eqref{Dupperbound} and \eqref{kappaineq}.
\end{proof}

\section{Diffusive bounds}
\label{s:diffusive_bounds}

\subsection{TSAW, general}
\label{ss:diffusive_bounds_TSAW}

We write the displacement $X(t)$ in the standard martingale + compensator decomposition:
\begin{align}
\label{martingale+compensator}
X(t)=N(t)+M(t)+ \int_0^t\ol\varphi(\eta(s))\,\d
s+ \int_0^t\wt\varphi(\eta(s))\,\d s.
\end{align}
Here, $N(t)$ is the martingale part due to the jump rates $\gamma$ and $M(t)$
is the martingale part due to the jump rates $w-\gamma$.

The compensators are
\begin{align}
\label{phi_bar}
&
\ol\varphi:\Omega\to\R^d,
&&
\ol\varphi_l(\omega)=
s(\omega(0)-\omega(e_l))-s(\omega(0)-\omega(-e_l)),
\\
\label{phi_tilde}
&
\wt\varphi:\Omega\to\R^d,
&&
\wt\varphi_l(\omega)=
r(\omega(0)-\omega(e_l))-r(\omega(0)-\omega(-e_l)).
\end{align}
Note that, since $s(\cdot)$ is \emph{even}, $\ol\varphi_l$, $l=1,\dots,d$ are
actually \emph{gradients}:
\begin{align}
\label{phi_bar_is_grad}
\ol\varphi_l=\nabla_l\psi_l
\quad\text{ where }\quad
\psi_l:\Omega\to\R, \quad \psi_l(\omega):=s(\omega(0)-\omega(-e_l)).
\end{align}

The \emph{diffusive lower bound} follows simply from \emph{ellipticity}
\eqref{ellipticity}. Indeed, it is straightforward that the martingale $N(t)$
in the decomposition \eqref{martingale+compensator} is uncorrelated with the
other terms. Hence, the lower bound in \eqref{diffusive_bounds}.

The main point is the \emph{diffusive upper bound} which is more subtle. Since
the martingale terms in \eqref{martingale+compensator} scale diffusively, we
only need to prove diffusive upper bound for the compensators. From standard
variational arguments, it follows (see e.g.\ \cite{komorowski_landim_olla_09},
\cite{olla_01} and \cite{sethuraman_varadhan_yau_00}) that
\begin{align}
\label{variational_bound}
\varlimsup_{t\to\infty}t^{-1}\expect{(\int_0^t\varphi(\eta(s))\,\d s)^2} \le
2(\varphi, S^{-1} \varphi).
\end{align}
In our particular case, from \eqref{symm_gen}, it follows that it is sufficient
to prove upper bounds on $(\ol\varphi, (-\Delta)^{-1} \ol\varphi)$ and
$(\wt\varphi, (-\Delta)^{-1} \wt\varphi)$. The first one drops out from
\eqref{phi_bar_is_grad}:
\begin{align}
\label{bound_on_phibar}
(\ol\varphi_l, (-\Delta)^{-1} \ol\varphi_l) =
(\nabla_l\psi_l, (-\Delta)^{-1} \nabla_l\psi_l) \le \norm{\psi_l}^2=
\expect{s(\omega(0)-\omega(e_l))^2}.
\end{align}
We need
\begin{align}
\label{bound_on_s_square}
\expect{s(\omega(0)-\omega(e_l))^2}<\infty.
\end{align}

In Lemma \ref{lemma:brascamp_lieb} below, we formulate a direct consequence of
Brascamp\,--\,Lieb inequality which will be used for proving
\eqref{bound_on_s_square} and also for the diffusive bound on the second
integral on the right-hand side of \eqref{martingale+compensator}.

Denote
\begin{align}
\label{Z_lambda}
Z(\lambda):=
\expect{\exp\{\lambda(\omega(0)-\omega(e))^2\}}\in[1,\infty].
\end{align}

\begin{lemma}
\label{lemma:brascamp_lieb}
For any smooth cylinder function $F:\Omega\to\R$ and $\lambda\in[0,c)$:
\begin{align}
\label{brascamp-lieb} & Z(\lambda) \expect{F(\omega)^2\exp\{ \lambda
(\omega(0)-\omega(e))^2\}}
\\
\notag & \hskip20mm \le\frac{1}{c-\lambda} Z(\lambda) \expect{\sum_{x,y\in\Z^d}
\partial_xF(\omega)
(-\Delta)^{-1}_{xy}
\partial_yF(\omega)
\exp\{\lambda(\omega(0)-\omega(e))^2\}}
\\
\notag & \hskip40mm +\expect{F(\omega)\exp\{ \lambda
(\omega(0)-\omega(e))^2\}}^2.
\end{align}
\end{lemma}

Lemma \ref{lemma:brascamp_lieb} follows directly from Brascamp\,--\,Lieb inequality as stated in e.g.\ Proposition 2.1 in \cite{bobkov_ledoux_00}. We omit its proof.

In order to prove \eqref{bound_on_s_square}, choose
$F(\omega)=\omega(0)-\omega(e)$ in \eqref{brascamp-lieb} and note that the second term on the right-hand side of the inequality vanishes. We get
\begin{align}
\label{zdot}
\frac{\d}{\d \lambda} Z(\lambda) \le
\frac{\beta}{c-\lambda}
Z(\lambda)
\end{align}
with some explicit constant $\beta<\infty$. Hence, for $\lambda\in[0,c)$,
\begin{align}
\label{Z-finite}
Z(\lambda)\le (1-(\lambda/c))^{-\beta}<\infty.
\end{align}
Now, \eqref{bound_on_s_square} follows from \eqref{s_small} and \eqref{Z-finite}.

In order to get
\begin{align}
\label{bound_on_phitilde}
(\wt\varphi_l, (-\Delta)^{-1}\wt\varphi_l)<\infty,
\end{align}
more argument is needed. To keep notation simple, we fix $l=1$ and drop the subscript. Denote
\begin{align}
\label{phi_tilde_cov}
C(x):=\expect{\wt\varphi(\omega)\wt\varphi(\tau_x\omega)}, \qquad
\wh C(p):=\sum_{x\in\Z^d}e^{ip\cdot x}C(x), \,\,\, p\in[-\pi,\pi]^d.
\end{align}
The bound \eqref{bound_on_phitilde}  is equivalent to the infrared bound
\begin{align}
\label{bound_on_C_hat_1}
\int_{[-\pi,\pi]^d} \frac{\wh C(p)}{\wh D(p)}\,\d p <\infty
\end{align}
where $\wh D$ is the Fourier transform of the lattice Laplacian, defined in \eqref{lat_Lap_Four}. Since $d\ge 3$, it is sufficient to prove
\begin{align}
\label{bound_on_C_hat_2}
\sup_{p\in [-\pi,\pi]^d} \abs{\wh C(p)} <\infty.
\end{align}

\begin{lemma}
\label{lemma:covbounds}
\newcounter{szaml4}
\begin{list}{(\alph{szaml4})}{\usecounter{szaml4}\setlength{\leftmargin}{1em}}

\item
Let $f:\R\to\R$ be smooth and denote
\begin{align}
C(x)
&:=
\cov{f(\omega(0)-\omega(e))}{f(\omega(x)-\omega(x+e))},
\\
C^{\prime}(x)
&:=
\cov{f^{\prime}(\omega(0)-\omega(e))}{f^{\prime}(\omega(x)-\omega(x+e))},
\\
m^{\prime}
&:=
\expect{f^{\prime}(\omega(0)-\omega(e))}.
\end{align}
Then
\begin{align}
\label{covbound1}
\sup_{p\in[-\pi,\pi]^d} \abs{\wh C(p)}
\le
(cd)^{-1} \sup_{p\in[-\pi,\pi]^d} \abs{\wh C'(p)} + c^{-1}(m')^2.
\end{align}

\item
Let
\begin{align}
\label{defCnm}
C_{nm}(x)
:=
\cov{(\omega(0)-\omega(e))^n}{(\omega(x)-\omega(x+e))^m}.
\end{align}
Then
\begin{align}
\label{covbound2}
\sup_{p\in[-\pi,\pi]^d}\abs{\wh C_{nm}(p)} \le
(Z(c/2))^2 n! m! \left(\frac2{c}\right)^{(n+m)/2}.
\end{align}

\item
If $r$ is an entire function and it satisfies condition \eqref{r_entire},
then
\begin{equation}
\sup_{p\in[-\pi,\pi]^d}\abs{\wh C(p)}<\infty.
\end{equation}
\end{list}
\end{lemma}

\begin{proof}
\newcounter{szaml5}
\begin{list}{(\alph{szaml5})}{\usecounter{szaml5}\setlength{\leftmargin}{1em}}

\item
We apply \eqref{brascamp-lieb} with $\lambda=0$ and
\begin{align}
\label{our_F}
F(\omega):=\sum_{x\in\Z^d}\alpha(x) f(\omega(x)-\omega(x+e))
\end{align}
where $\alpha:\Z^d\to\R$ is finitely supported and $\sum_{z\in\Z^d}\alpha(z)=0$. Straightforward computations yield
\begin{align}
\label{bound1}
\sum_{x,y\in\Z^d}\alpha(x)C(x-y)\alpha(y)
\le
c^{-1} \sum_{x,y\in\Z^d} \alpha(x) \Gamma(x-y) \big(C^{\prime}(x-y)+(m^{\prime})^2\big) \alpha(y)
\end{align}
where $\Gamma$ is the matrix
\begin{equation}
\label{matrix_R} \Gamma:=\nabla_1(-\Delta^{-1})\nabla_1
\end{equation}
well-defined in any dimension. Its Fourier transform is
\begin{equation}
\label{R_hat}
\wh \Gamma(p)=\frac{1-\cos p_1}{\wh D(p)}.
\end{equation}

The bound \eqref{bound1} is equivalent to
\begin{align}
\label{bound2}
\wh C(p) \le c^{-1}\left(\wh \Gamma*\wh C'(p) + (m')^2 \wh \Gamma(p) \right).
\end{align}
Convolution is meant periodically in $[-\pi,\pi]^d$. Hence,
\begin{align}
\notag
\sup_{p\in[-\pi,\pi]^d} \abs{\wh C(p)}
&
\le
c^{-1} \sup_{p\in[-\pi,\pi]^d} \abs{\wh C'(p)} \int_{[-\pi,\pi]^d}\wh \Gamma(p)\,\d p +
c^{-1}(m')^2 \sup_{p\in[-\pi,\pi]^d} \wh\Gamma(p)
\\
\label{bound3}
&
=
(cd)^{-1} \sup_{p\in[-\pi,\pi]^d} \abs{\wh C'(p)}+
c^{-1} (m')^2.
\end{align}

\item
We apply \eqref{covbound1} to the function $f(u)=u^n$ and use the notation
\begin{equation}
m_n:=\expect{(\omega(0)-\omega(e))^n}
\end{equation}
to get
\begin{equation}
\sup_{p\in[-\pi,\pi]^d} \abs{\wh C_{nn}(p)}
\le
(cd)^{-1} n^2 \sup_{p\in[-\pi,\pi]^d} \abs{\wh C_{n-1,n-1}(p)} + c^{-1} n^2 m_{n-1}^2.
\end{equation}
Induction on $n$ yields
\begin{equation}
\label{covbound3}
\sup_{p\in[-\pi,\pi]^d} \abs{\wh C_{nn}(p)}
\le
\sum_{k=1}^n
\frac{(n!)^2}{((n-k)!)^2} \frac{m_{n-k}^2}{c^kd^{k-1}}.
\end{equation}

By \eqref{Z-finite}, we have the finiteness of
\begin{equation}
\expect{\exp\left( {c} (\omega(0)-\omega(e))^2\right)/2}
=
Z({c}/2)
<\infty.
\end{equation}
Hence, by expanding the exponential,
\begin{equation}
m_n= \expect{(\omega(0)-\omega(e))^n} \le Z({c}/2) \frac{2^{n/2}\lfloor
n/2\rfloor!}{c^{n/2}}
\end{equation}
follows. We neglect the fact that the odd moments are $0$ by symmetry.
Combining the last inequality with \eqref{covbound3}, we obtain
\begin{align}
\label{covbound4}
\sup_{p\in[-\pi,\pi]^d} \abs{\wh C_{nn}(p)}
&\le
(Z({c}/2))^2(n!)^2 c^{-n}
\sum_{k=1}^n
\frac{(\lfloor (n-k)/2\rfloor!)^2}{((n-k)!)^2}
\frac{2^{n-k}}{d^{k-1}}
\\
&\le
(Z({c}/2))^2(n!)^2 \left(\frac2{c}\right)^n,
\end{align}
which proves \eqref{covbound2} for $n=m$. The constant $2/c$ is far from
optimal here, but the order $(n!)^2$ is the best one can get with this
argument.

The general case $n\not=m$ follows by Schwarz's inequality.

\item
By power series expansion of the entire function $r$, and using
\eqref{phi_tilde}, \eqref{covbound2} and \eqref{r_entire}, one gets
\begin{align}
\abs{\wh C(p)} &\le 4 \sum_{n=0}^\infty \sum_{m=0}^\infty
\frac{\abs{r^{(n)}(0)}}{n!} \frac{\abs{r^{(m)}(0)}}{m!} \abs{\wh C_{nm}(p)}
\\
\notag
&
\le
4(Z({c}/2))^2
\left(\sum_{n=0}^\infty
\abs{r^{(n)}(0)} \left(\frac2{c}\right)^{n/2}\right)^2<\infty.
\end{align}

\end{list}
\end{proof}

\begin{remark}
The diffusive upper bound is much simpler in the special Gaussian case. There,
since $r(u)=u$ the compensator function $\wt\varphi$ defined in
\eqref{phi_tilde} is also a gradient (easily seen) and thus, by an argument
identical to \eqref{bound_on_phibar}, we obtain the desired upper bound.
Something similar happens in the SRBP case in the next subsection.
\end{remark}

\subsection{SRBP, Gaussian}
\label{ss:diffusive_bounds_SRBP}

Now the martingale + compensator decomposition has the form
\begin{align}
\label{martingale+compensator_SRBP} X(t)=B(t) + \int_0^t \varphi(\eta_s)\,\d s
\end{align}
where
\begin{align}
\label{compansator_SRBP}
\varphi_l(\omega):=\partial_l\omega(0).
\end{align}

First, we prove a diffusive lower bound. For $s,t\in\R$ with $s<t$, let
\begin{equation}
\label{Mstdef}
M(s,t):=X(t)-X(s)-\int_s^t \varphi(\eta(u))\,\d
u=B(t)-B(s).
\end{equation}

\begin{lemma}
\label{lemma:forwbackw}
\newcounter{szaml7}
\begin{list}
{(\arabic{szaml7})}{\usecounter{szaml7}\setlength{\leftmargin}{1em}}
\item
Fix $s\in\R$. The process $[s,\infty)\ni t\mapsto M(s,t)$ is a forward
martingale with respect to the forward filtration $\{\cF_{(-\infty,t]}:t\ge
s\}$ of the process $t\mapsto\eta(t)$.

\item
Fix $t\in\R$. The process $(-\infty,t]\ni s\mapsto M(s,t)$ is a backward
martingale with respect to the backward filtration $\{\cF_{[s,\infty)}:s\le
t\}$ of the process $t\mapsto\eta(t)$.
\end{list}
\end{lemma}

\begin{proof}
There is nothing to prove about the first statement: the integral on the right-hand side of \eqref{Mstdef} was chosen exactly so that it compensates the conditional expectation of the infinitesimal increments of $X(t)$.

We turn to the second statement, which does need a proof. This consists of the following ingredients:
\begin{enumerate}[(1)]

\item
The displacements are preserved on the flipped backward trajectories
$t\mapsto\wt\eta(t)$ defined in \eqref{rev-flip} for the TSAW:
\begin{equation}
\wt X(t)-\wt X(s)=X(-t)-X(-s).
\end{equation}

\item
The forward process $t\mapsto\eta(t)$ and flipped backward process $t\mapsto\wt\eta(t)$ are identical in law (Yaglom reversibility).

\item
The function $\omega\mapsto\varphi(\omega)$ is odd with respect to the flip-map $\omega\mapsto -\omega$.

\end{enumerate}

\noindent
Putting these facts together (in this order), we obtain
\begin{align}
\lim_{h\to0}(-h)^{-1}\condexpect{X(s-h)-X(s)}{\cF_{[s,\infty)}} & =
\lim_{h\to0}h^{-1}\condexpect{-\wt X(-s+h)+\wt X(-s)}{\wt
\cF_{(-\infty,-s]}}\notag
\\
\label{bwmart} & = -\varphi(\wt\eta(-s)) = \varphi(\eta(s)).
\end{align}

\end{proof}

From Lemma \ref{lemma:forwbackw}, it follows directly that for any $s<t$, the
random variables $M(s,t)$ and $\int_s^t \varphi(\eta(u))\,\d u$ are
\emph{uncorrelated}, and therefore,
\begin{align}
\label{variancesum}
\expect{(X(t)-X(s))^2}
&=
\expect{(M(s,t))^2}+
\expect{\big(\int_s^t \varphi(\eta(u))\,\d u\big)^2}
\\
\notag
&=
(t-s) +
\expect{\big(\int_s^t \varphi(\eta(u))\,\d u\big)^2}.
\end{align}
Hence, the lower bound in \eqref{bounds}.

Next, we turn to the diffusive upper bound. We use the variational upper bound
\eqref{variational_bound} and note that, in our present case, $S=\abs{\Delta}$
and $\varphi$ is given in \eqref{compansator_SRBP}. Hence, by straightforward
computations, we get
\begin{equation}
\label{ourbound}
(\varphi_l, S^{-1}\varphi_l)
=
\int_{\mathbb{R}^d}
\frac{p_l^2}{\abs{p}^2}\,\wh V(p)\,\d p<\infty.
\end{equation}

\hfill\qed

\section{Checking the graded sector condition}
\label{s:check_gsc}

\subsection{TSAW}
\label{ss:check_gsc_TSAW}

We are ready to prove the second part of Theorem \ref{thm:diffusive_limit}. We
have to show that the martingale approximation of Theorem KV is  valid for the
integrals in on the right-hand side of \eqref{martingale+compensator}. We apply
the \emph{graded sector condition} formulated in Theorem GSC with  $D=\gamma
\abs{\Delta}$ and the operators $S$ and $A$ given in graded form in
\eqref{opsgrading} and \eqref{opagrading}. \eqref{Dellipticity} clearly holds
and \eqref{Dupperbound} was already proved in Section \ref{s:diffusive_bounds}.
We still need to verify conditions \eqref{diagbound}, \eqref{oddbound} and
\eqref{evenbound}.

Checking \eqref{evenbound} is straightforward: If $s(u)$ is an even polynomial
of degree $2q$, then using in turn \eqref{nonameopnorm}, \eqref{astaropnorm}
and \eqref{aopnorm}, we obtain
\begin{align}
\label{checkevenbound}
\norm{|\Delta|^{-1/2}\nabla_{-e} s(a_e+a_e^*) \nabla_e|\Delta|^{-1/2}\upharpoonright_{\cH_n} }
\le
\norm{ s(a_e+a_e^*) \upharpoonright_{\cH_n} }
\le
c n^q +C
\end{align}
with the constant $c$ proportional to the leading coefficient in the polynomial
$s(u)$ and $C<\infty$. Hence, if $q=2$ (that is: $s(u)$ quartic polynomial) and
the leading coefficient is sufficiently small, then \eqref{evenbound} follows.
The bound \eqref{diagbound} with $\kappa=2$ also drops out from
\eqref{checkevenbound}.

Finally, we check \eqref{oddbound}. By \eqref{nonameopnorm},
\begin{align}
\label{checkoddbound1} \norm{|\Delta|^{-1/2} a_{-e}^* \nabla_e|\Delta|^{-1/2}
\upharpoonright_{\cH_n} } \le \norm{|\Delta|^{-1/2}
a_e^*\upharpoonright_{\cH_n} }.
\end{align}
We prove
\begin{align}
\label{checkoddbound2} \norm{|\Delta|^{-1/2} a_e^*\upharpoonright_{\cH_n} } \le
C n^{1/2}
\end{align}
with some finite constant $C$.

For $\wh u\in \cH_n$,
\begin{multline}
\abs{\Delta}^{-1/2} a^*_e \wh u(p_1,\dots,p_{n+1})
\\
=\frac1{\sqrt{n+1}} \frac{1}{\sqrt{\wh D(\sum_{m=1}^{n+1} p_m)}}
\sum_{m=1}^{n+1}\left(e^{ip_m\cdot e}-1\right) \wh
u(p_1,\dots,\cancel{p_{m}},\dots,p_{n+1}).
\end{multline}
Hence,
\begin{align}
\label{bouTSAW} & \norm{\abs{\Delta}^{-1/2} a^*_e \wh u}^2
\\
\notag &= \frac1{n+1} \int_{(-\pi,\pi]^{d(n+1)}} \frac1{\wh D(\sum_{m=1}^{n+1}
p_m)}
\\
\notag &\hskip3cm\times \abs{\sum_{m=1}^{n+1} \left(e^{ip_m\cdot e}-1\right)
\wh u(p_1,\!\dots,\cancel{p_{m}},\dots,p_{n+1})}^2
\prod_{m=1}^{n+1}\frac{1}{\wh D(p_m)}\,\d p_1\dots\d p_{n+1}
\\
\notag &\le (n+1) \int_{(-\pi,\pi]^{d(n+1)}} \frac1{\wh D(\sum_{m=1}^{n+1}
p_m)}
\\
\notag &\hskip3cm\times \abs{e^{ip_{n+1}\cdot e}-1}^2 \abs{\wh
u(p_1,\dots,p_{n})}^2 \prod_{m=1}^{n+1}\frac{1}{\wh D(p_m)}\,\d p_1\dots\d
p_{n+1}
\\
\notag &= (n+1) \int_{(-\pi,\pi]^{dn}} \abs{\wh u(p_1,\dots,p_{n})}^2
\prod_{m=1}^{n}\frac{1}{\wh D(p_m)}
\\
\notag &\hskip3cm\times \left( \int_{(-\pi,\pi]^d} \frac{\abs{e^{ip_{n+1}\cdot
e}-1}^2} {\wh D(p_{n+1})} \frac{1} {\wh D(\sum_{m=1}^{n+1} p_m)} \d p_{n+1}
\right) \d p_1\dots\d p_{n}.
\end{align}
Schwarz's inequality and symmetry was used. Note that on the right-hand side, for the innermost term, since $d\ge3$, we have
\begin{align}
\int_{(-\pi,\pi]^d}
\frac{\abs{e^{ip_{n+1}\cdot e}-1}^2} {\wh D(p_{n+1})}
\frac{1} {\wh D(\sum_{m=1}^{n+1} p_m)}
\d p_{n+1}
\le C^2.
\end{align}
Hence,
\begin{equation}
\norm{\abs{\Delta}^{-1/2} a^*_e \wh u}^2 \le C^2 (n+1) \norm{\wh u}^2,
\end{equation}
and \eqref{checkoddbound2} follows. This proves \eqref{oddbound}.\qed

\subsection{SRBP}
\label{ss:check_gsc_SRBP}

As a first remark, note that condition \eqref{oddbound} follows from
\begin{equation}
\label{gsc2} \norm{S^{-1/2}A_+S^{-1/2}\upharpoonright_{\cH_n}}\le Cn^{\alpha}
\end{equation}
with $\alpha<1$ where the operator $S^{-1/2}A_+S^{-1/2}\upharpoonright_{\cH_n}$
is meant as first defined on a dense subspace of $\cH_n$ and extended by
continuity. In our case, the dense subspace will be $\wh \cC_n$ specified in
\eqref{coredefin} and
\begin{equation}
S^{-1/2}A_+S^{-1/2}
=
\sum_{l=1}^d \abs{\Delta}^{-1/2} a^*_l \nabla_l\abs{\Delta}^{-1/2}
\end{equation}
The operators $\nabla_l\abs{\Delta}^{-1/2}$ map the subspaces $\wh \cC_n$ to
themselves and are bounded, see \eqref{nonameopnormSRBP}. In order to bound the
norm of the operator $\abs{\Delta}^{-1/2} a^*_l:\cH_n\to\cH_{n+1}$, let $\wh
u\in\wh \cC_{n}$, then
\begin{equation}
\abs{\Delta}^{-1/2} a^*_l \wh u(p_1,\dots,p_{n+1}) = \frac{i}{\sqrt{n+1}}
\frac{1}{\abs{\sum_{m=1}^{n+1} p_m}} \sum_{m=1}^{n+1}p_{ml}\wh
u(p_1,\dots,\cancel{p_m},\dots,p_n).
\end{equation}
The rest of the calculations is a direct analogue to its TSAW counterpart, and is left to the reader. It turns out that
\begin{equation}
\norm{\abs{\Delta}^{-1/2} a^*_l\upharpoonright_{\cH_n}} \le C \sqrt{n+1},
\end{equation}
so in fact \eqref{gsc2} follows with $\alpha=1/2$. \qed

\bigskip

\begin{ack}
BT thanks illuminating discussions with Marek Biskup, Stefano Olla and Benedek Valk\'o, and the kind hospitality of the Mittag\,--\,Leffler Insitute, Stockholm, where part of this work was done. The work of all authors was partially supported by OTKA (Hungarian National Research Fund) grant K 60708.
\end{ack}

\end{document}